\documentclass [twoside,reqno,   12pt] {amsart}

\usepackage{amsfonts}
\usepackage{amssymb}
\usepackage{a4}
\usepackage{color}

\newtheorem{thm}{Theorem}[section]

\newtheorem{lem}[thm]{Lemma}
\newtheorem{prop}[thm]{Proposition}

\theoremstyle{definition}

\numberwithin{equation}{section}

\renewcommand{\Re}{\hbox{Re}\,}
\renewcommand{\Im}{\hbox{Im}\,}

\newcommand{\C}{\mathbb{C}}

\newcommand{\R}{\mathbb{R}}

\newcommand{\supp}{\operatorname{supp}}

\def\hat{\widehat}
\def\tilde{\widetilde}
\def \bfo {\begin {eqnarray*} }
\def \efo {\end {eqnarray*} }
\def \ba {\begin {eqnarray*} }
\def \ea {\end {eqnarray*} }
\def \beq {\begin {eqnarray}}
\def \eeq {\end {eqnarray}}
\def \supp {\hbox{supp }}

\def \p {\partial}

\def\hat{\widehat}
\def\tilde{\widetilde}
\def \bfo {\begin {eqnarray*} }
\def \efo {\end {eqnarray*} }
\def \ba {\begin {eqnarray*} }
\def \ea {\end {eqnarray*} }
\def \beq {\begin {eqnarray}}
\def \eeq {\end {eqnarray}}
\def \supp {\hbox{supp }}

\def \p {\partial}

\begin{document}

 \title[Magnetic Schr\"odinger operator with bounded potentials]{Uniqueness in an inverse boundary problem for a magnetic Schr\"odinger operator with a bounded magnetic potential}

\author[Krupchyk]{Katsiaryna Krupchyk}

\address
        {K. Krupchyk, Department of Mathematics and Statistics \\
         University of Helsinki\\
         P.O. Box 68 \\
         FI-00014   Helsinki\\
         Finland}

\email{katya.krupchyk@helsinki.fi}

\author[Uhlmann]{Gunther Uhlmann}

\address
       {G. Uhlmann, Department of Mathematics\\
       University of Washington\\
       Seattle, WA  98195-4350\\
       and
       Department of Mathematics\\ 
       340 Rowland Hall\\
        University of California\\
        Irvine, CA 92697-3875\\
       USA}
\email{gunther@math.washington.edu}

\maketitle

\begin{abstract} 

We show that the knowledge of the set of the Cauchy data on the boundary of a bounded open set in $\R^n$, $n\ge 3$,  for the magnetic Schr\"odinger operator with $L^\infty$ magnetic and electric potentials determines the magnetic field and electric potential inside the set uniquely.  The proof is based on a Carleman estimate for the magnetic Schr\"odinger operator with a gain of two derivatives.

\end{abstract}

\section{Introduction and statement of result}

Let $\Omega\subset \R^n$, $n\ge 3$, be a bounded open set, and let $u\in C^\infty_0(\Omega)$.  We consider the magnetic Schr\"odinger operator,
\begin{align*}
L_{A,q}(x,D)&u(x):=\sum_{j=1}^n (D_j+A_j(x))^2u(x)+q(x)u(x)\\
&=-\Delta u(x) +A(x)\cdot Du(x) +D\cdot (A(x)u(x))   +((A(x))^2+q(x))u(x),
\end{align*}
where $D=i^{-1}\nabla$, $A\in L^\infty(\Omega,\C^n)$ is the magnetic potential, and $q\in L^\infty(\Omega,\C)$ is the electric potential. We have $Au\in L^\infty(\Omega,\C^n)\cap\mathcal{E}'(\Omega,\C^n)$, and therefore, 
\[
L_{A,q}:C^\infty_0(\Omega)\to H^{-1}(\R^n)\cap\mathcal{E}'(\Omega)
\]
is a bounded operator. Here $\mathcal{E}'(\Omega)=\{v\in \mathcal{D}'(\Omega):\supp(v)\textrm{ is compact}\}$.

Let us now introduce the Cauchy data for an $H^1(\Omega)$ solution  $u$ to the equation 
\begin{equation}
\label{eq_1_1}
L_{A,q}u=0\quad \textrm{in}\quad \Omega,
\end{equation}
in the sense of distributions.  First, following \cite{Astala_Paivarinta, Salo_diss}, we define the trace space of the space $H^1(\Omega)$ as the quotient space $H^1(\Omega)/H_0^1(\Omega)$. The associated trace map 
$T:H^1(\Omega)\to H^1(\Omega)/H_0^1(\Omega)$, $Tu=[u]$, is the quotient map.  Here $H^1_0(\Omega)$ is the closure of $C_0^\infty(\Omega)$ with respect to the $H^1(\Omega)$--topology. 

Notice that if $\Omega$ has a Lipschitz boundary, then the space $H^1(\Omega)/H_0^1(\Omega)$ can be naturally identified with  the Sobolev space $H^{1/2}(\p \Omega)$. Indeed, in this case the kernel of the continuous surjective map $H^1(\Omega)\to H^{1/2}(\p \Omega)$, $u\mapsto u|_{\p \Omega}$ is precisely $H^1_0(\Omega)$, see \cite[Theorems 3.37 and 3.40]{McLean_book}. 

For $u\in H^1(\Omega)$ satisfying \eqref{eq_1_1}, we can define $N_{A,q}u$, formally given by $N_{A,q}u=(\p_\nu u+i(A\cdot \nu)u)|_{\p \Omega}$,  as an element of the dual space $(H^1(\Omega)/H^1_0(\Omega))'$ as follows.  For $[g]\in H^1(\Omega)/H^1_0(\Omega)$, we set 
\begin{equation}
\label{eq_trace_boundary_arbitrary}
\begin{aligned}
(N_{A,q}u, [g])_\Omega:=\int_{\Omega} (\nabla u\cdot\nabla g + iA\cdot (u\nabla g-g\nabla u)+ (A^2+q)u g)\,dx.
\end{aligned}
\end{equation}
As $u$ is a solution to \eqref{eq_1_1},  $N_{A,q}u$ is a well-defined element of $(H^1(\Omega)/H^1_0(\Omega))'$.

We define the set of the Cauchy data for solutions of the magnetic Schr\"odinger equation as follows, 
\[
C_{A,q}:=\{(Tu,N_{A,q}u): u\in H^1(\Omega)\textrm{ and }  L_{A,q}u=0 \textrm{ in } \Omega\}.
\]
The inverse boundary value problem for the magnetic Schr\"odinger operator $L_{A,q}$ is to determine $A$ and $q$ in $\Omega$ from the set of the Cauchy data $C_{A,q}$.

Similarly to \cite{Sun_1993}, there is an obstruction to uniqueness in this problem given by the following gauge equivalence of the set of the Cauchy data: if $\psi\in W^{1,\infty}$ in a neighborhood of $\overline{\Omega}$ and $\psi|_{\p \Omega}=0$, then 
$C_{A,q}=C_{A+\nabla\psi,q}$,
see Lemma \ref{lem_obstruction_to_uniqueness} below. Hence, the map $A\mapsto A+\nabla\psi$ transforms the magnetic potential into a gauge equivalent one but preserves the induced magnetic field $dA$,
which is defined by 
\[
dA=\sum_{1\le j<k\le n}(\p_{x_j}A_k-\p_{x_k}A_j)dx_j\wedge dx_k,
\]
in the sense of distributions.  Here $A=(A_1,\dots,A_n)$. In view of this and of the fact that 
the magnetic field is a physically observable quantity,  one may hope to recover the magnetic field $dA$ and the electric potential $q$ in $\Omega$ from the set of the Cauchy data  $C_{A,q}$. 

As it has been shown by several authors, the knowledge of the set of the Cauchy data $C_{A,q}$ for the magnetic Schr\"odinger operator $L_{A,q}$ does determine the magnetic field $dA$ and the electric potential $q$ in $\Omega$ uniquely, 
under certain regularity assumptions on $A$ and $q$. In \cite{Sun_1993}, this result was established for magnetic potentials in $W^{2,\infty}$, satisfying a smallness condition, and $L^\infty$ electric potentials.  In \cite{NakSunUlm_1995}, the smallness condition 
was eliminated for smooth magnetic and electric potentials, and for compactly supported $C^2$ magnetic potentials and $L^\infty$ electric potentials.  The uniqueness results were subsequently extended to 
$C^1$ magnetic potentials in \cite{Tolmasky_1998}, to some less regular but small potentials in \cite{Panchenko_2002}, and to Dini continuous magnetic potentials in \cite{Salo_diss}.

The purpose of this paper is to extend the uniqueness result to the case of magnetic Schr\"odinger operators with magnetic potentials that are of class  $L^\infty$. Our main result is as follows. 

\begin{thm}
\label{thm_main}

Let $\Omega\subset \R^n$, $n\ge 3$, be a bounded open set, and let $A_1,A_2\in L^\infty(\Omega,\C^n)$ and $q_1,q_2\in L^\infty(\Omega,\C)$. If $C_{A_1,q_1}=C_{A_2,q_2}$, then $dA_1=dA_2$ and $q_1=q_2$ in $\Omega$. 
\end{thm}

Notice in particular that in Theorem \ref{thm_main} no regularity assumptions on the boundary of $\Omega$ are required.

The key ingredient in the proof of Theorem \ref{thm_main} is a construction of  complex geometric optics solutions for the magnetic Schr\"odinger operator $L_{A,q}$ with $A\in L^\infty(\Omega,\C^n)$ and $q\in L^\infty(\Omega,\C)$.  When constructing such solutions, we shall first derive a Carleman estimate for the magnetic Schr\"odinger operator $L_{A,q}$, with a gain of two derivatives, which  is based on the corresponding Carleman estimate for the Laplacian, obtained in \cite{Salo_Tzou_2009}.  Another crucial observation, which allows us to handle the case of $L^\infty$ magnetic potentials is that it is in fact sufficient to approximate the magnetic potential by a sequence of smooth vector fields, in the $L^2$ sense.  

We would also like to mention that another important inverse boundary value problem, for which the issues of regularity have been studied extensively, is Calder\'on's problem for the conductivity equation, see \cite{Calderon}.  The unique identifiability of $C^2$ conductivities from boundary measurements was established in \cite{Syl_Uhl_1987}. The regularity assumptions were relaxed to conductivities having  $3/2+\varepsilon$ derivatives in \cite{Brown_1996}, and the uniqueness 
for conductivities having exactly $3/2$ derivatives was obtained in \cite{Paiv_Pan_Uhl}, see also \cite{Brown_Torres_2003}. In \cite{Green_Lassas_Uhlmann_2003}, uniqueness for conormal conductivities in $C^{1+\varepsilon}$ was shown.  The recent work \cite{Hab_Tataru} proves a uniqueness result for Calder\'on's problem with conductivities 
of class $C^1$ and with Lipschitz continuous conductivities, which are close to the identity in a suitable sense. 

The paper is organized as follows.    Section \ref{sec_CGO} contains the construction of complex geometric optics solutions for the magnetic Schr\"odinger operator with $L^\infty$ magnetic and electric potentials.  The proof of Theorem \ref{thm_main} is then completed in Section \ref{sec_proof_main}.

\section{Construction of complex geometric optics solutions} 

\label{sec_CGO}

Let $\Omega\subset\R^n$, $n\ge 3$, be a bounded open set.  Following \cite{DKSU_2007, KenSjUhl2007}, we shall use the method of Carleman estimates to construct complex geometric optics solutions for the magnetic Schr\"odinger equation $L_{A,q}u=0$ in $\Omega$,  with $A\in  L^\infty(\Omega,\C^n)$ and $q\in L^\infty(\Omega,\C)$.

Let us start by recalling the Carleman estimate for the semiclassical Laplace operator $-h^2\Delta$ with a gain of two derivatives, established in   \cite{Salo_Tzou_2009}, see also \cite{KenSjUhl2007}. 
Here $h>0$ is a small semiclassical parameter. 
Let $\tilde \Omega$ be an open set in $\R^n$ such that $ \Omega\subset\subset\tilde \Omega$ and let
$\varphi\in C^\infty(\tilde \Omega,\R)$.  Consider the conjugated operator 
\[
P_\varphi=e^{\frac{\varphi}{h}}(-h^2\Delta) e^{-\frac{\varphi}{h}},
\]
with the semiclassical principal symbol
\[
p_\varphi(x,\xi)=\xi^2+2i\nabla \varphi\cdot \xi-|\nabla \varphi|^2, \quad x\in \tilde {\Omega},\quad  \xi\in \R^n. 
\]
We have for $(x,\xi)\in \overline{\Omega}\times \R^n$, $|\xi|\ge C\gg 1$, that $|p_\varphi(x,\xi)|\sim |\xi|^2$ so that $P_\varphi$ is elliptic at infinity, in the semiclassical sense. 
Following \cite{KenSjUhl2007}, we say that $\varphi$ is a limiting Carleman weight for $-h^2\Delta$ in $\tilde \Omega$, if $\nabla \varphi\ne 0$ in $\tilde \Omega$ and the Poisson bracket of $\Re p_\varphi$ and $\Im p_\varphi$ satisfies, 
\[
\{\Re p_\varphi,\Im p_\varphi\}(x,\xi)=0 \quad \textrm{when}\quad p_\varphi(x,\xi)=0, \quad (x,\xi)\in \tilde{\Omega}\times \R^n. 
\]
Examples of limiting Carleman weights  are  linear weights $\varphi(x)=\alpha\cdot x$, $\alpha\in \R^n$, $|\alpha|=1$, and logarithmic weights $\varphi(x)=\log|x-x_0|$,  with $x_0\not\in \tilde \Omega$.  In this paper we shall only use the linear weights. 

Our starting point is the following result due to \cite{Salo_Tzou_2009}.
\begin{prop}
Let $\varphi$ be a limiting Carleman weight for the semiclassical Laplacian on $\tilde \Omega$, and let $\varphi_\varepsilon=\varphi+\frac{h}{2\varepsilon}\varphi^2$.  Then 
for $0<h\ll \varepsilon\ll 1$ and $s\in\R$, we have
\begin{equation}
\label{eq_Carleman_lap}
\frac{h}{\sqrt{\varepsilon}}\|u\|_{H^{s+2}_{\emph{scl}}(\R^n)}\le C\|e^{\varphi_\varepsilon/h}(-h^2\Delta)e^{-\varphi_\varepsilon/h}u\|_{H^s_{\emph{scl}}(\R^n)},
\quad C>0,
\end{equation}
 for all $u\in C^\infty_0(\Omega)$.  
\end{prop}
Here 
\[
\|u\|_{H^s_{\textrm{scl}}(\R^n)}=\|\langle hD \rangle^s u\|_{L^2(\R^n)},\quad \langle\xi  \rangle=(1+|\xi|^2)^{1/2},
\]
is the natural semiclassical norm in the Sobolev space  $H^s(\R^n)$, $s\in\R$.

Next we shall derive a Carleman estimate for the magnetic Schr\"odinger operator $L_{A,q}$ with $A\in L^\infty(\Omega,\C^n)$ and $q\in L^\infty(\Omega,\C)$.  To that end we shall use the estimate \eqref{eq_Carleman_lap} with $s=-1$, and with $\varepsilon>0$ being sufficiently small but fixed, i.e. independent of $h$.  We have the following result. 
 
\begin{prop}
Let $\varphi\in C^\infty(\tilde\Omega,\R)$ be a limiting Carleman weight for the semiclassical Laplacian on $\tilde \Omega$, and assume that $A\in L^\infty(\Omega,\C^n)$ and $q\in L^\infty(\Omega,\C)$. Then 
for $0<h\ll 1$, we have 
\begin{equation}
\label{eq_Carleman_schr}
h\|u\|_{H^{1}_{\emph{scl}}(\R^n)}\le C\|e^{\varphi/h}(h^2L_{A,q})e^{-\varphi/h}u\|_{H^{-1}_{\emph{scl}}(\R^n)},
\end{equation}
 for all $u\in C^\infty_0(\Omega)$.  
\end{prop}

\begin{proof}
In order to prove the estimate \eqref{eq_Carleman_schr} it will be convenient to use the following characterization of the semiclassical norm in the Sobolev space $H^{-1}(\R^n)$, 
\begin{equation}
\label{eq_charac_H_-}
\|v\|_{H_{\textrm{scl}}^{-1}(\R^n)}=\sup_{0\ne \psi\in C^\infty_0(\R^n)}\frac{|\langle v,\psi\rangle_{\R^n} |}{\|\psi\|_{H_{\textrm{scl}}^{1}(\R^n)}},
\end{equation}
where $\langle \cdot,\cdot\rangle_{\R^n}$ is the distribution duality on $\R^n$. 

Let $\varphi_\varepsilon=\varphi+\frac{h}{2\varepsilon}\varphi^2$ be the convexified weight with $\varepsilon>0$ such that $0<h\ll \varepsilon\ll 1$, and let $u \in C^\infty_0(\Omega)$.  Then for all $0\ne \psi\in C^\infty_0(\R^n)$, we have 
\begin{align*}
|\langle e^{\varphi_\varepsilon/h} h^2 A\cdot D(e^{-\varphi_\varepsilon/h}u), \psi \rangle_{\R^n}|&\le 
\int_{\R^n}\bigg|hA\cdot \bigg(-u\bigg(1+\frac{h}{\varepsilon}\varphi\bigg)D\varphi+hDu\bigg)\psi\bigg|dx\\
&\le \mathcal{O}(h)\|u\|_{H^1_{\textrm{scl}}(\R^n)}\|\psi\|_{H^1_{\textrm{scl}}(\R^n)}.
\end{align*}
We also obtain that
\begin{align*}
|\langle e^{\varphi_\varepsilon/h} h^2 D\cdot( A e^{-\varphi_\varepsilon/h}u), \psi \rangle_{\R^n}|&\le \int_{\R^n}|h^2A e^{-\varphi_\varepsilon/h}u\cdot D(e^{\varphi_\varepsilon/h}\psi)|dx\\
&\le \mathcal{O}(h)\|u\|_{H^1_{\textrm{scl}}(\R^n)}\|\psi\|_{H^1_{\textrm{scl}}(\R^n)}.
\end{align*} 
Hence, using \eqref{eq_charac_H_-}, we get
\begin{equation}
\label{eq_2_3}
\|e^{\varphi_\varepsilon/h} h^2 A\cdot D(e^{-\varphi_\varepsilon/h}u) + e^{\varphi_\varepsilon/h} h^2 D\cdot( A e^{-\varphi_\varepsilon/h}u) \|_{H_{\textrm{scl}}^{-1}(\R^n)}\le \mathcal{O}(h)\|u\|_{H^1_{\textrm{scl}}(\R^n)}.
\end{equation}
Notice that the implicit constant in \eqref{eq_2_3} only depends on $\|A\|_{L^\infty(\Omega)}$, $\|\varphi\|_{L^\infty(\Omega)}$ and $\|D\varphi\|_{L^\infty(\Omega)}$.  Now choosing $\varepsilon >0$  sufficiently small but fixed, i.e. independent of $h$, we conclude from the estimate \eqref{eq_Carleman_lap} with $s=-1$ and the estimate \eqref{eq_2_3} that for all $h>0$ small enough, 
\begin{equation}
\label{eq_2_4}
\begin{aligned}
\|e^{\varphi_\varepsilon/h}(-h^2\Delta)e^{-\varphi_\varepsilon/h}u&+e^{\varphi_\varepsilon/h} h^2 A\cdot D(e^{-\varphi_\varepsilon/h}u) + e^{\varphi_\varepsilon/h} h^2 D\cdot( A e^{-\varphi_\varepsilon/h}u) \|_{H_{\textrm{scl}}^{-1}(\R^n)}\\
&\ge \frac{h}{C}\|u\|_{H^1_{\textrm{scl}}(\R^n)},\quad C>0. 
\end{aligned}
\end{equation}
Furthermore, the estimate
\[
\|h^2(A^2+q)u\|_{H^{-1}_{\textrm{scl}}(\R^n)}\le \mathcal{O}(h^2)\|u\|_{H^1_{\textrm{scl}}(\R^n)} 
\]
and the estimate \eqref{eq_2_4} imply that for all $h>0$ small enough, 
\[
\|e^{\varphi_\varepsilon/h}(h^2L_{A,q})e^{-\varphi_\varepsilon/h}u\|_{H_{\textrm{scl}}^{-1}(\R^n)}\ge \frac{h}{C}\|u\|_{H^1_{\textrm{scl}}(\R^n)},\quad C>0. 
\]
Using that 
\[
e^{-\varphi_\varepsilon/h}u=e^{-\varphi/h}e^{-\varphi^2/(2\varepsilon)}u,
\]
we obtain \eqref{eq_Carleman_schr}. The proof is complete.

\end{proof}

Let $\varphi\in C^\infty(\tilde\Omega,\R)$ be a limiting Carleman weight for $-h^2\Delta$ and set $L_\varphi=e^{\varphi/h}(h^2L_{A,q})e^{-\varphi/h}$.
Then we have
\[
\langle L_\varphi u,\overline{v}\rangle_{\Omega}= \langle  u,\overline{ L_\varphi^* v}\rangle_{\Omega}, \quad u,v\in C^\infty_0(\Omega),
\]
where  $L_\varphi^*=e^{-\varphi/h}(h^2L_{\overline{A},\overline{q}})e^{\varphi/h}$ is the formal adjoint of $L_\varphi$ and $\langle \cdot,\cdot\rangle_{\Omega}$ is the distribution duality on $\Omega$. We have
\[
L_\varphi^*:C^\infty_0(\Omega)\to H^{-1}(\R^n)\cap\mathcal{E}'(\Omega)
\]
is bounded, and the estimate \eqref{eq_Carleman_schr} holds for $L_\varphi^*$, since $-\varphi$ is a limiting Carleman weight as well.

To construct complex geometric optics solutions for the magnetic Schr\"odinger operator we  need to convert the Carleman estimate  \eqref{eq_Carleman_schr} for $L_\varphi^*$ into 
the following solvability result. The proof is essentially well-known, and is included here for the convenience of the reader. We shall write
\begin{align*}
&\|u\|_{H^1_{\textrm{scl}}(\Omega)}^2=\|u\|_{L^2(\Omega)}^2+\|hDu\|_{L^2(\Omega)}^2,\\
&\|v\|_{H^{-1}_{\textrm{scl}}(\Omega)}=\sup_{0\ne \psi\in C_0^\infty(\Omega)}\frac{|\langle v,\psi\rangle_{\Omega}|}{\|\psi\|_{H^1_{\textrm{scl}}(\Omega)}}.
\end{align*}

\begin{prop}
\label{prop_solvability}
Let $A\in L^\infty(\Omega,\C^n)$, $q\in L^\infty(\Omega,\C)$, and let $\varphi$ be a limiting Carleman weight for the semiclassical Laplacian on $\tilde \Omega$. If $h>0$ is small enough, then for any $v\in H^{-1}(\Omega)$, there is a solution $u\in H^1(\Omega)$ of the equation
\[
e^{\varphi/h}(h^2L_{A,q})e^{-\varphi/h}u=v\quad\textrm{in}\quad \Omega,
\]
which satisfies
\[
\|u\|_{H^1_{\emph{scl}}(\Omega)}\le \frac{C}{h}\|v\|_{H^{-1}_{\emph{scl}}(\Omega)}.
\]
\end{prop}

\begin{proof}
Let $v\in H^{-1}(\Omega)$ and let us consider the following complex linear functional,
\[
L: L_\varphi^* C_0^\infty(\Omega)\to \C, \quad L_\varphi^* w \mapsto \langle w, \overline{v}\rangle_\Omega. 
\]
By the Carleman estimate \eqref{eq_Carleman_schr} for $L_\varphi^*$, the map $L$ is well-defined.  Let $w\in C_0^\infty(\Omega)$. Then we have
\begin{align*}
|L(L_\varphi^* w)|=|\langle w, \overline{v}\rangle_\Omega|&\le \|w\|_{H^1_{\textrm{scl}}(\R^n)}\|v\|_{H^{-1}_{\textrm{scl}}(\Omega)}\\
&\le \frac{C}{h}\|v\|_{H^{-1}_{\textrm{scl}}(\Omega)}\|L_\varphi^* w\|_{H^{-1}_{\textrm{scl}}(\R^n)}.
\end{align*}
By the Hahn-Banach theorem, we may extend $L$ to a linear continuous functional $\tilde L$ on $H^{-1}(\R^n)$, without increasing its norm. 
By the Riesz representation theorem, there exists $u\in H^1(\R^n)$ such that for all $\psi\in H^{-1}(\R^n)$,
\[
\tilde L(\psi)=\langle \psi,\overline{u}\rangle_{\R^n}, \quad \textrm{and}\quad \|u\|_{H^1_{\textrm{scl}}(\R^n)}\le \frac{C}{h}\|v\|_{H^{-1}_{\textrm{scl}}(\Omega)}. 
\]
Let us now show that $L_\varphi u=v$ in $\Omega$. To that end, let $w\in C_0^\infty(\Omega)$. Then 
\[
\langle L_\varphi u,\overline{w}\rangle_\Omega=\langle u,\overline{L_\varphi^*w}\rangle_{\R^n} =\overline{\tilde L(L_\varphi^*w)}=\overline{\langle w,\overline{v}\rangle_\Omega}=\langle v,\overline{w}\rangle_\Omega. 
\]
The proof is complete. 
\end{proof}

Let $A\in L^\infty(\Omega,\C^n)$. We shall extend $A$ to $\R^n$ by defining it to be zero in  $\R^n\setminus\Omega$,  and denote this extension by the same letter. Then $A\in (L^\infty\cap\mathcal{E}')(\R^n,\C^n)\subset L^p(\R^n,\C^n)$, $1\le p\le \infty$. 

Let $\Psi_\tau(x)=\tau^{-n}\Psi(x/\tau)$, $\tau>0$, be the usual mollifier with $\Psi\in C^\infty_0(\R^n)$, $0\le \Psi\le 1$, and 
$\int \Psi dx=1$.  Then $A^\sharp=A*\Psi_\tau\in C_0^\infty(\R^n,\C^n)$ and 
\begin{equation}
\label{eq_flat_est}
\|A-A^\sharp\|_{L^2(\R^n)}=o(1), \quad \tau\to 0.
\end{equation}
A direct computation shows that 
\begin{equation}
\label{eq_flat_est_2}
 \|\p^\alpha A^\sharp\|_{L^\infty(\R^n)}=\mathcal{O}(\tau^{-|\alpha|}),\quad \tau\to 0, \quad \textrm{for all}\quad \alpha,\quad |\alpha|\ge 0.
\end{equation}

We shall now construct complex geometric optics solutions for the magnetic Schr\"odinger equation 
\begin{equation}
\label{eq_2_6}
L_{A,q}u=0\quad \textrm{in} \quad \Omega,
\end{equation} 
with $A\in L^\infty(\Omega,\C^n)$ and $q\in L^\infty(\Omega,\C)$, using the solvability result of Proposition \ref{prop_solvability} and 
the approximation \eqref{eq_flat_est}. Complex geometric optics solutions are solutions of the form,
\begin{equation}
\label{eq_2_7}
u(x,\zeta;h)=e^{x\cdot\zeta/h} (a(x,\zeta;h)+r(x,\zeta;h)),
\end{equation}
where $\zeta\in\C^n$, $\zeta\cdot\zeta=0$, $|\zeta|\sim 1$,  $a$ is a smooth amplitude, $r$ is a correction term, and $h>0$ is a small parameter.

It will be convenient to introduce the following bounded operator, 
\[
m_A: H^1(\Omega)\to H^{-1}(\Omega),\quad  m_A(u)=D\cdot(A u),
\]
where the distribution $m_A(u)$ is given by
\[
\langle m_A(u),v \rangle_\Omega=-\int_{\Omega} Au\cdot Dv dx, \quad v\in C_0^\infty(\Omega). 
\]
Let us conjugate $h^2L_{A,q}$ by $e^{x\cdot\zeta/h}$. First, let us compute $e^{-x\cdot\zeta/h}\circ h^2m_A \circ e^{x\cdot \zeta/h}$. When $u\in H^1(\Omega)$ and $v\in C^\infty_0(\Omega)$, we get 
\begin{align*}
\langle e^{-x\cdot\zeta/h}h^2m_A (e^{x\cdot \zeta/h}u),v\rangle_\Omega &=-\int_{\Omega} h^2 A e^{x\cdot \zeta/h}u \cdot D(e^{-x\cdot\zeta/h} v)dx\\
&=-\int_\Omega(hi\zeta\cdot Auv+h^2Au\cdot Dv)dx,
\end{align*}
and therefore, 
\[
e^{-x\cdot\zeta/h}\circ h^2m_A \circ e^{x\cdot \zeta/h}= -hi\zeta\cdot A+h^2 m_A. 
\]
Furthermore, we obtain that 
\begin{align*}
e^{-x\cdot\zeta/h}\circ (-h^2\Delta) \circ e^{x\cdot \zeta/h}=-h^2\Delta -2ih \zeta\cdot D,\\
e^{-x\cdot\zeta/h}\circ h^2(A\cdot D)\circ  e^{x\cdot \zeta/h}=h^2A\cdot D- hi\zeta\cdot A. 
\end{align*}
Hence, we have
\begin{equation}
\label{eq_2_8}
e^{-x\cdot\zeta/h}  \circ h^2 L_{A,q} \circ e^{x\cdot \zeta/h}=-h^2\Delta -2ih \zeta\cdot D +h^2A\cdot D- 2 hi\zeta\cdot A+ h^2 m_A + h^2(A^2+q). 
\end{equation}
We shall consider $\zeta$ depending slightly on $h$, i.e. $\zeta=\zeta_0+\zeta_1$ with $\zeta_0$ being independent of $h$ and $\zeta_1=\mathcal{O}(h)$ as $h\to 0$.  We also assume that $|\Re\zeta_0|=|\Im\zeta_0|=1$. Then we write \eqref{eq_2_8} as follows,
\begin{align*}
e^{-x\cdot\zeta/h} \circ h^2  L_{A,q}\circ  e^{x\cdot \zeta/h}= &  -h^2\Delta -2ih \zeta_0\cdot D -2ih \zeta_1\cdot D +h^2A\cdot D- 2 hi\zeta_0\cdot A^\sharp\\
&- 2 hi\zeta_0\cdot (A-A^\sharp) - 2 hi\zeta_1\cdot A+ h^2 m_A + h^2(A^2+q). 
\end{align*}

In order that \eqref{eq_2_7} be a solution of \eqref{eq_2_6}, we require that 
\begin{equation}
\label{eq_2_9}
\zeta_0\cdot Da +\zeta_0\cdot A^\sharp a=0\quad\textrm{in}\quad \R^n,
\end{equation}
and 
\begin{equation}
\label{eq_2_10}
\begin{aligned}
e^{-x\cdot\zeta/h} h^2 L_{A,q}& e^{x\cdot \zeta/h}r= -(- h^2\Delta a  +h^2A\cdot D a + h^2 m_A (a)+ h^2(A^2+q)a)\\
& +2ih \zeta_1\cdot Da  + 2 hi\zeta_0\cdot (A-A^\sharp)a +2 hi\zeta_1\cdot Aa=:g \quad\textrm{in}\quad \Omega. 
\end{aligned}
\end{equation}
The equation \eqref{eq_2_9} is the first transport equation and  one looks for its solution in the form $a=e^{\Phi^\sharp}$, where $\Phi^\sharp$ solves the equation
\begin{equation}
\label{eq_2_9_phi}
\zeta_0\cdot\nabla \Phi^\sharp + i\zeta_0\cdot A^\sharp=0\quad \textrm{in}\quad \R^n. 
\end{equation}
As $\zeta_0\cdot\zeta_0=0$ and $|\Re\zeta_0|=|\Im\zeta_0|=1$,  the operator   $N_{\zeta_0}:=\zeta_0\cdot\nabla$ is the $\bar\p$--operator in suitable linear coordinates. Let us introduce an inverse operator defined by
\[
(N_{\zeta_0}^{-1}f)(x) =\frac{1}{2\pi}\int_{\R^2} \frac{f(x-y_1\Re\zeta_0 -y_2\Im\zeta_0)}{y_1+iy_2}dy_1dy_2,\quad f\in C_0(\R^n).
\]
We have the following result, see \cite[Lemma 4.6]{Salo_diss}. 
\begin{lem}
\label{lem_salo_1}
 Let $f\in W^{k,\infty}(\R^n)$, $k\ge 0$, with $\emph{\supp}(f)\subset B(0,R)$. Then $\Phi=N_{\zeta_0}^{-1} f\in W^{k,\infty}(\R^n)$ satisfies $N_{\zeta_0}\Phi=f$ in $\R^n$, and we have 
\begin{equation}
\label{eq_salo_1}
\|\Phi\|_{W^{k,\infty}(\R^n)}\le C\|f\|_{W^{k,\infty}(\R^n)}, 
\end{equation}
where $C=C(R)$.  If $f\in C_0(\R^n)$, then $\Phi\in C(\R^n)$. 
\end{lem}
Thanks to Lemma \ref{lem_salo_1},  the function $\Phi^\sharp(x,\zeta_0; \tau):=N_{\zeta_0}^{-1}(-i\zeta_0\cdot A^\sharp)\in C^\infty(\R^n)$ satisfies the equation \eqref{eq_2_9_phi}. 
Furthermore, the estimates \eqref{eq_flat_est_2} and \eqref{eq_salo_1} imply that 
\begin{equation}
\label{eq_ampl_est}
\|\p^\alpha \Phi^\sharp \|_{L^\infty(\R^n)}\le C_\alpha \tau^{-|\alpha|}, \quad \textrm{for all}\quad \alpha, \quad |\alpha|\ge 0.   
\end{equation}

Owing to \cite[Lemma 3.1]{Syl_Uhl_1987}, we have the following result, where we use the norms
\[
\|f\|_{L^2_\delta(\R^n)}^2=\int_{\R^n}(1+|x|^2)^{\delta}|f(x)|^2dx. 
\]
\begin{lem}
\label{lem_SU}
Let $-1<\delta<0$ and let $f\in L^2_{\delta+1}(\R^n)$. Then there exists a constant $C>0$, independent of $\zeta_0$, such that 
\[
\|N^{-1}_{\zeta_0} f\|_{L^2_\delta(\R^n)}\le C\|f\|_{L^2_{\delta+1}(\R^n)}.
\]
\end{lem}

Setting $\Phi(\cdot,\zeta_0):=N_{\zeta_0}^{-1}(-i\zeta_0\cdot A)\in L^\infty(\R^n)$, it follows from Lemma \ref{lem_SU} and the estimate \eqref{eq_flat_est} that $\Phi^\sharp (\cdot,\zeta_0;\tau)$ converges to $\Phi(\cdot,\zeta_0)$ in $L^2_{\textrm{loc}}(\R^n)$ as $\tau\to 0$.

Let us turn now to the equation \eqref{eq_2_10}. First notice that the right hand side $g$Ê of \eqref{eq_2_10} belongs to $H^{-1}(\Omega)$ and we would like to  estimate $\|g\|_{H^{-1}_{\textrm{scl}}(\Omega)}$.  To that end,  let $0\ne \psi\in C_0^\infty(\Omega)$. Then using \eqref{eq_ampl_est} and the fact that $\zeta_1=\mathcal{O}(h)$, we get by the Cauchy--Schwarz inequality, 
\begin{align*}
&|\langle h^2\Delta a, \psi \rangle_\Omega|\le \mathcal{O} (h^2/\tau^{2}) \|\psi\|_{L^2(\Omega)}\le \mathcal{O} (h^2/\tau^{2}) \|\psi\|_{H^{1}_{\textrm{scl}}(\Omega)}, \\
&|\langle h^2 A\cdot Da,\psi\rangle_\Omega|\le \mathcal{O} (h^2/\tau) \|\psi\|_{H^{1}_{\textrm{scl}}(\Omega)},\\
& |\langle 2ih\zeta_1\cdot Da,\psi\rangle_\Omega|\le \mathcal{O} (h^2/\tau) \|\psi\|_{H^{1}_{\textrm{scl}}(\Omega)},\\
& |\langle 2hi\zeta_1\cdot Aa,\psi \rangle_\Omega|\le \mathcal{O} (h^2) \|\psi\|_{H^{1}_{\textrm{scl}}(\Omega)}.
\end{align*}
Using \eqref{eq_flat_est} and  \eqref{eq_ampl_est}, we have
\begin{align*}
|\langle 2hi\zeta_0 \cdot(A-A^\sharp)a,\psi\rangle_\Omega |&\le \mathcal{O}(h)\|a\|_{L^\infty(\R^n)}\|A-A^\sharp\|_{L^2(\Omega)}\|\psi\|_{L^2(\Omega)}\\
&\le \mathcal{O}(h)o_{\tau\to 0}(1) \|\psi\|_{H^1_{\textrm{scl}}(\Omega)}.
\end{align*}
With the help of  \eqref{eq_flat_est}, \eqref{eq_flat_est_2}, and  \eqref{eq_ampl_est}, we obtain that 
\begin{align*}
|\langle h^2m_A(a),& \psi \rangle_\Omega|\le \bigg| \int_\Omega h^2 A^\sharp a\cdot D\psi dx\bigg| + \bigg|\int_\Omega h^2 (A-A^\sharp) a\cdot D\psi dx\bigg|\\
&\le \bigg| \int_\Omega h^2 (D\cdot (A^\sharp a)) \psi dx\bigg| +\mathcal{O}(h)\|A-A^\sharp\|_{L^2(\Omega)}\|hD\psi\|_{L^2(\Omega)}\\
&\le (\mathcal{O}(h^2/\tau) +\mathcal{O}(h)o_{\tau\to 0}(1)) \|\psi\|_{H^1_{\textrm{scl}}(\Omega)}. 
\end{align*}
We also have $\|h^2(A^2+q)a\|_{L^2(\Omega)}\le \mathcal{O}(h^2)$. Thus, from the above estimates, we conclude that 
\[
\|g\|_{H^{-1}_{\textrm{scl}}(\Omega)}\le  \mathcal{O} (h^2/\tau^{2}) + \mathcal{O}(h)o_{\tau\to 0}(1).
\]
Choosing now $\tau=h^\sigma$ with some $\sigma$,  $0<\sigma<1/2$, we get
\begin{equation}
\label{eq_2_11}
\|g\|_{H^{-1}_{\textrm{scl}}(\Omega)}=o(h) \quad \textrm{as}\quad h\to 0. 
\end{equation}
Thanks to Proposition \ref{prop_solvability} and \eqref{eq_2_11}, for $h>0$ small enough, there exists a solution $r\in H^1(\Omega)$ of \eqref{eq_2_10} such that $\|r\|_{H^1_{\textrm{scl}}(\Omega)}=o(1)$ as $h\to 0$. 

The discussion led in this section can be summarized  in the following proposition. 
\begin{prop}
\label{prop_cgo_solutions}
Let $\Omega\subset \R^n$, $n\ge 3$, be a bounded open set.  Let $A\in L^\infty(\Omega,\C^n)$, $q\in L^\infty(\Omega,\C)$, and let $\zeta\in \C^n$ be such that $\zeta\cdot\zeta=0$, $\zeta=\zeta_0+\zeta_1$ with $\zeta_0$ being independent of $h>0$, $|\emph{\Re}\zeta_0|=|\emph{\Im} \zeta_0|=1$, and $\zeta_1=\mathcal{O}(h)$ as $h\to 0$.    Then for all $h>0$ small enough, there exists a solution $u(x,\zeta;h)\in H^1(\Omega)$ to the magnetic Schr\"odinger equation $L_{A,q}u=0$ in $\Omega$, of the form
\[
u(x,\zeta;h)=e^{x\cdot\zeta/h}(e^{\Phi^\sharp(x,\zeta_0;h)}+r(x,\zeta;h)).
\] 
The function  $\Phi^\sharp (\cdot,\zeta_0;h)\in C^\infty(\R^n)$ satisfies 
$\|\p^\alpha \Phi^\sharp\|_{L^\infty(\R^n)}\le C_\alpha h^{-\sigma|\alpha|}$, $0<\sigma<1/2$,   
for all $\alpha$, $|\alpha|\ge 0$, and $\Phi^\sharp (\cdot,\zeta_0;h)$ converges  to $\Phi(\cdot,\zeta_0):=N_{\zeta_0}^{-1}(-i\zeta_0\cdot A)\in L^\infty(\R^n)$ in $L^2_{\emph{\textrm{loc}}}(\R^n)$ as $h\to 0$. Here we have extended $A$ by  zero to $\R^n\setminus\Omega$.  The remainder $r$ is such that $\|r\|_{H^1_{\emph{\textrm{scl}}}(\Omega)}=o(1)$ as $h\to 0$. 
\end{prop}

\section{Proof of Theorem \ref{thm_main}}

\label{sec_proof_main}

Let us begin by recalling the following auxiliary, essentially well-known, result which shows that the set of the Cauchy data for the magnetic Schr\"odinger operator remains unchanged if the gradient of a function, vanishing along the boundary, is added to the magnetic potential, see \cite[Lemma 4.1]{Salo_diss}, \cite{Sun_1993}. 

\begin{lem}
\label{lem_obstruction_to_uniqueness}
Let $\Omega\subset\R^n$ be a bounded open set,  let $A\in L^\infty(\Omega,\C^n)$, $q\in L^\infty(\Omega, \C)$, and let $\psi\in W^{1,\infty}$ in a neighborhood of $\overline{\Omega}$. Then we have 
\begin{equation}
\label{eq_conj_lem}
e^{-i\psi}\circ L_{A,q}\circ e^{i\psi}=L_{A+\nabla \psi,q}.
\end{equation}
If furthermore, $\psi|_{\p \Omega}=0$ then 
\begin{equation}
\label{eq_conj_lem_2}
C_{A,q}=C_{A+\nabla\psi,q}. 
\end{equation}

\end{lem}

\begin{proof}
Let us notice first that the assumption that $\psi\in W^{1,\infty}$ in a neighborhood of $\overline{\Omega}$ implies that $\psi$ is Lipschitz continuous on $\overline{\Omega}$, so that $\psi|_{\p \Omega}$ is well-defined pointwise. 

Since \eqref{eq_conj_lem} follows by a direct computation, only \eqref{eq_conj_lem_2} has to be established. To that end, let $u\in H^1(\Omega)$ be a solution to $L_{A,q}u=0$ in $\Omega$. Then 
$e^{-i\psi}u\in H^1(\Omega)$ satisfies $L_{A+\nabla \psi,q}(e^{-i\psi}u)=0$ in $\Omega$.  Let us show that $T(e^{-i\psi}u)=Tu$. In other words, we have to check that 
\begin{equation}
\label{eq_conj_lem_3}
u(e^{-i\psi}-1)\in H^1_0(\Omega).
\end{equation} 
Since the function $e^{-i\psi}-1$ is Lipschitz continuous on $\overline{\Omega}$ and vanishes along $\p \Omega$, we have $|e^{-i\psi(x)}-1|\le C d(x)$ for any $x\in \Omega$ and some constant $C>0$. Here $d(x)$ is the distance from $x$ to the boundary of $\Omega$. Then \eqref{eq_conj_lem_3} follows from the following fact: if $v\in H^1(\Omega)$ and $v/d\in L^2(\Omega)$,  then $v\in H^1_0(\Omega)$, see \cite[Theorem 3.4, p. 223]{Edmunds_Evans_book}.

Let us now show that $N_{A+\nabla\psi,q}(e^{-i\psi}u)=N_{A,q}u$. To that end, first as above, one observes that for $g\in H^{1}(\Omega)$, we have $[g]=[e^{i\psi}g]$. Thus, 
\[
(N_{A+\nabla\psi,q}(e^{-i\psi}u),[g] )_\Omega=(N_{A+\nabla\psi,q}(e^{-i\psi}u),[e^{i\psi}g] )_{\Omega}=( N_{A,q}(u),[g] )_\Omega,
\]
for any $[g]\in H^1(\Omega)/H^1_0(\Omega)$, and therefore, $C_{A,q}\subset C_{A+\nabla\psi,q}$. The proof is complete. 
\end{proof}

The first step in the proof of Theorem \ref{thm_main} is the derivation of the following integral identity based on the fact that $C_{A_1,q_1}=C_{A_2,q_2}$, see also \cite[Lemma 4.3]{Salo_diss}. 
\begin{prop}
\label{prop_eq_int_identity}
Let $\Omega\subset \R^n$, $n\ge 3$,  be a bounded open set.  Assume that $A_1,A_2\in L^\infty(\Omega,\C^n)$ and $q_1,q_2\in L^\infty(\Omega,\C)$. If $C_{A_1,q_1}=C_{A_2,q_2}$, then the following integral identity 
\begin{equation}
\label{eq_int_identity}
\int_\Omega i(A_1-A_2)\cdot (u_1\nabla \overline{u_2}-\overline{u_2}\nabla u_1)dx+\int_\Omega(A_1^2-A_2^2+q_1-q_2)u_1\overline{u_2}dx=0 
\end{equation}
holds for any $u_1,u_2\in H^1(\Omega)$ satisfying $L_{A_1,q_1}u_1=0$ in $\Omega$ and $L_{\overline{A_2},\overline{q_2}}u_2=0$ in $\Omega$, respectively.  
\end{prop}

\begin{proof}
Let $u_1, u_2\in H^1(\Omega)$ be  solutions to $L_{A_1,q_1}u_1=0$ in $\Omega$ and 
$L_{\overline{A_2},\overline{q_2}}u_2=0$ in $\Omega$, respectively.  Then the fact that $C_{A_1,q_1}=C_{A_2,q_2}$ implies that there is $v_2\in H^1(\Omega)$ satisfying $L_{A_2,q_2}v_2=0$ in $\Omega$ such that 
\[
Tu_1=Tv_2\quad\textrm{and}\quad N_{A_1,q_1}u_1=N_{A_2,q_2}v_2. 
\]
This together with \eqref{eq_trace_boundary_arbitrary} shows that 
\[
(N_{A_1,q_1}u_1 ,[\overline{u_2}] )_\Omega=( N_{A_2,q_2}v_2, [\overline{u_2}])_\Omega=\overline{( N_{\overline{A_2},\overline{q_2}}u_2, [\overline{v_2}])_\Omega}
=\overline{( N_{\overline{A_2},\overline{q_2}}u_2, [\overline{u_1}])_\Omega}.
\]
Then the integral identity \eqref{eq_int_identity} follows from the definition \eqref{eq_trace_boundary_arbitrary} of $N_{A_1,q_1}u_1$ and $N_{\overline{A_2},\overline{q_2}}u_2$.  The proof is complete. 
\end{proof}

We shall use the integral identity \eqref{eq_int_identity} with $u_1$ and $u_2$ being complex geometric optics solutions for the magnetic Schr\"odinger equations in $\Omega$. To construct such solutions,  let $\xi,\mu_1,\mu_2\in\R^n$ be such that $|\mu_1|=|\mu_2|=1$ and $\mu_1\cdot\mu_2=\mu_1\cdot\xi=\mu_2\cdot\xi=0$. 
Similarly to \cite{Sun_1993}, we set 
\begin{equation}
\label{eq_zeta_1_2}
\zeta_1=\frac{ih\xi}{2}+\mu_1 + i\sqrt{1-h^2\frac{|\xi|^2}{4}}\mu_2 , \quad 
\zeta_2=-\frac{ih\xi}{2}-\mu_1+i\sqrt{1-h^2\frac{|\xi|^2}{4}}\mu_2,
\end{equation}
so that $\zeta_j\cdot\zeta_j=0$, $j=1,2$, and $(\zeta_1+\overline{\zeta_2})/h=i\xi$. Here $h>0$ is a small enough semiclassical parameter.  Moreover, $\zeta_1= \mu_1+ i\mu_2+\mathcal{O}(h)$ and $\zeta_2= -\mu_1+ i\mu_2+\mathcal{O}(h)$ as $h\to 0$. 

By Proposition \ref{prop_cgo_solutions},  for all $h>0$ small enough, there exists a solution $u_1(x,\zeta_1;h)\in H^1(\Omega)$ to the magnetic Schr\"odinger equation $L_{A_1,q_1}u_1=0$ in $\Omega$, of the form
\begin{equation}
\label{eq_u_1}
u_1(x,\zeta_1;h)=e^{x\cdot\zeta_1/h}(e^{\Phi_1^\sharp(x,\mu_1+i\mu_2;h)}+r_1(x,\zeta_1;h)),
\end{equation}
where $\Phi_1^\sharp(\cdot,\mu_1+i\mu_2;h)\in C^\infty(\R^n)$ satisfies the estimate 
\begin{equation}
\label{eq_est_phi_1}
\|\p^\alpha \Phi_1^\sharp\|_{L^\infty(\R^n)}\le C_\alpha h^{-\sigma|\alpha|},\quad 0<\sigma<1/2,
\end{equation}
for all $\alpha$, $|\alpha|\ge 0$, $\Phi_1^\sharp(\cdot,\mu_1+i\mu_2;h)$ converges to 
\begin{equation}
\label{eq_phi_1_def}
\Phi_{1}(\cdot,\mu_1+i\mu_2):=N_{\mu_1+i\mu_2}^{-1}
(-i(\mu_1+i\mu_2)\cdot A_1)\in L^\infty(\R^n)
\end{equation}
 in $L^2_{\textrm{loc}}(\R^n)$ as $h\to 0$, 
and 
\begin{equation}
\label{eq_est_r_1}
\|r_1\|_{H^1_{\textrm{scl}}(\Omega)}=o(1)\quad \textrm{as}\quad h\to 0.
\end{equation}
Similarly, for all $h>0$ small enough, there exists a solution $u_2(x,\zeta_2;h)\in H^1(\Omega)$ to the magnetic Schr\"odinger equation $L_{\overline{A_2},\overline{q_2}}u_2=0$ in $\Omega$, of the form
\begin{equation}
\label{eq_u_2}
u_2(x,\zeta_2;h)=e^{x\cdot\zeta_2/h}(e^{\Phi_2^\sharp(x,-\mu_1+i\mu_2;h)}+r_2(x,\zeta_2;h)),
\end{equation}
where $\Phi_2^\sharp(\cdot,-\mu_1+i\mu_2;h)\in C^\infty(\R^n)$ satisfies the estimate 
\begin{equation}
\label{eq_est_phi_2}
\|\p^\alpha \Phi_2^\sharp\|_{L^\infty(\R^n)}\le C_\alpha h^{-\sigma|\alpha|},\quad 0<\sigma<1/2,
\end{equation}
for all $\alpha$, $|\alpha|\ge 0$. Furthermore, $\Phi_2^\sharp(\cdot,-\mu_1+i\mu_2;h)$ converges to 
\begin{equation}
\label{eq_phi_2_def}
\Phi_{2}(\cdot,-\mu_1+i\mu_2):=N_{-\mu_1+i\mu_2}^{-1}
(-i(-\mu_1+i\mu_2)\cdot \overline{A_2})\in L^\infty(\R^n)
\end{equation}
 in $L^2_{\textrm{loc}}(\R^n)$ as $h\to 0$, 
and 
\begin{equation}
\label{eq_est_r_2}
\|r_2\|_{H^1_{\textrm{scl}}(\Omega)}=o(1)\quad \textrm{as}\quad h\to 0.
\end{equation}

We shall next  substitute $u_1$ and $u_2$, given by \eqref{eq_u_1} and \eqref{eq_u_2}, into the integral identity \eqref{eq_int_identity}, multiply it by $h$, and let $h\to 0$. We first compute
\begin{align*}
h u_1\nabla\overline{u_2}=&\overline{\zeta_2}e^{ix\cdot\xi}(e^{\Phi_1^\sharp+\overline{\Phi_2^\sharp}}+e^{\Phi_1^\sharp}\overline{r_2}+r_1e^{\overline{\Phi_2^\sharp}}+r_1\overline{r_2})\\
&+he^{ix\cdot\xi}(e^{\Phi_1^\sharp}\nabla e^{\overline{\Phi_2^\sharp}} + e^{\Phi_1^\sharp}\nabla \overline{r_2} + r_1\nabla e^{\overline{\Phi_2^\sharp}} +r_1  \nabla \overline{r_2}).
\end{align*}
Recall that $\overline{\zeta_2}=-\mu_1-i\mu_2+\mathcal{O}(h)$. We shall show that 
\[
(\mu_1+i\mu_2)\cdot\int_\Omega (A_1-A_2) e^{ix\cdot\xi}e^{\Phi_1^\sharp+\overline{\Phi_2^\sharp}}dx\to (\mu_1+i\mu_2)\cdot\int_\Omega (A_1-A_2) e^{ix\cdot\xi}e^{\Phi_1+\overline{\Phi_2}}dx,
\] 
as $h\to 0$, where $\Phi_1$ and $\Phi_2$ are defined by \eqref{eq_phi_1_def} and \eqref{eq_phi_2_def}, respectively.   To that end, we have
\begin{align*}
\bigg| (\mu_1+i\mu_2)\cdot\int_\Omega (A_1-A_2) e^{ix\cdot\xi}\big(e^{\Phi_1^\sharp+\overline{\Phi_2^\sharp}}- e^{\Phi_1+\overline{\Phi_2}}\big)dx \bigg|\le C\big\|e^{\Phi_1^\sharp+\overline{\Phi_2^\sharp}}- e^{\Phi_1+\overline{\Phi_2}}\big\|_{L^2(\Omega)}\\
\le C\|\Phi_1^\sharp+\overline{\Phi_2^\sharp}-\Phi_1-\overline{\Phi_2}\|_{L^2(\Omega)}\to 0,
\end{align*}
as $h\to 0$. Here we have used the inequality 
\begin{equation}
\label{eq_complex_exp}
|e^z-e^w|\le |z-w|e^{\textrm{max}(\textrm{Re}\, z,\textrm{Re}\, w)},\quad z,w\in \C,
\end{equation}
obtained by integration of $e^z$ from $z$ to $w$, and the fact that $\Phi_j, \Phi_j^\sharp\in L^\infty(\R^n)$, $j=1,2$, and $\|\Phi_j^\sharp\|_{L^\infty(\R^n)}\le C$ uniformly in $h$.

Now using the estimates \eqref{eq_est_phi_1}, \eqref{eq_est_r_1}, \eqref{eq_est_phi_2} and \eqref{eq_est_r_2}, we get
\begin{align*}
\bigg|\int_\Omega & i(A_1-A_2)\cdot \overline{\zeta_2} e^{ix\cdot\xi}(e^{\Phi_1^\sharp}\overline{r_2}+r_1e^{\overline{\Phi_2^\sharp}}+r_1\overline{r_2})dx\bigg|\\
&\le C\|A_1-A_2\|_{L^\infty}
(\|e^{\Phi_1^\sharp}\|_{L^2}\|\overline{r_2}\|_{L^2}+\|r_1\|_{L^2}\|e^{\overline{\Phi_2^\sharp}}\|_{L^2}+\|r_1\|_{L^2}\|\overline{r_2}\|_{L^2})=o(1),
\end{align*}
as $h\to 0$. We also  obtain that 
\begin{align*}
\bigg|\int_\Omega h i(A_1-A_2)\cdot e^{ix\cdot\xi}(e^{\Phi_1^\sharp}\nabla e^{\overline{\Phi_2^\sharp}} + e^{\Phi_1^\sharp}\nabla \overline{r_2} + r_1\nabla e^{\overline{\Phi_2^\sharp}} +r_1  \nabla \overline{r_2})dx\bigg|\\
\le \mathcal{O}(h)(h^{-\sigma}+ h^{-1}o(1)+ o(1)h^{-\sigma}+o(1)h^{-1})=o(1),
\end{align*}
as $h\to 0$. Here $0<\sigma<1/2$.  Furthermore, 
\begin{align*}
\bigg|h\int_\Omega(A_1^2-A_2^2+q_1-q_2)e^{ix\cdot\xi}(e^{\Phi_1^\sharp+\overline{\Phi_2^\sharp}}+e^{\Phi_1^\sharp}\overline{r_2}+r_1e^{\overline{\Phi_2^\sharp}}+r_1\overline{r_2})dx\bigg|=\mathcal{O}(h),
\end{align*}
as $h\to 0$. Hence, substituting  $u_1$ and $u_2$, given by \eqref{eq_u_1} and \eqref{eq_u_2}, into the integral identity \eqref{eq_int_identity}, multiplying it by $h$, and letting $h\to 0$, we get
\begin{equation}
\label{eq_with_phases_R_n}
(\mu_1+i\mu_2)\cdot\int_{\R^n} (A_1-A_2) e^{ix\cdot\xi} e^{\Phi_{1}(x,\mu_1+i\mu_2)+\overline{\Phi_{2}(x,-\mu_1+i\mu_2)}}dx=0,
\end{equation}
where 
\begin{align*}
\Phi_{1}=N_{\mu_1+i\mu_2}^{-1}(- i(\mu_1+i\mu_2)\cdot A_1) \in L^\infty(\R^n),\\
\Phi_{2}= N_{-\mu_1+i\mu_2}^{-1} ( - i(-\mu_1+i\mu_2)\cdot \overline{A_2})\in L^\infty(\R^n).
\end{align*}
Notice that the integration in \eqref{eq_with_phases_R_n} is extended to all of $\R^n$, since $A_1=A_2=0$ on $\R^n\setminus{\Omega}$.

The next step is to remove the  function $e^{\Phi_{1}+\overline{\Phi_{2}}}$ in the integral \eqref{eq_with_phases_R_n}. First using the following properties of the Cauchy transform, 
\[
\overline{N_{\zeta}^{-1}f}=N_{\overline{\zeta}}^{-1}\overline{f},\quad  N_{-\zeta}^{-1}f=-N_\zeta^{-1}f,
\]
we see that 
\begin{equation}
\label{eq_sum_phase}
\Phi_{1}+\overline{\Phi_{2}}=N_{\mu_1+i\mu_2}^{-1}(- i(\mu_1+i\mu_2)\cdot (A_1-A_2)).  
\end{equation}

We have the following result.   
\begin{prop} 
\label{prop_eskin_ralston}
Let  $\xi,\mu_1,\mu_2\in\R^n$, $n\ge 3$,  be such that $|\mu_1|=|\mu_2|=1$ and $\mu_1\cdot\mu_2=\mu_1\cdot\xi=\mu_2\cdot\xi=0$. 
Let $W\in (L^\infty\cap \mathcal{E}')(\R^n, \C^n)$ and $\phi=N_{\mu_1+i\mu_2}^{-1}(- i(\mu_1+i\mu_2)\cdot W)$. Then 
\begin{equation}
\label{eq_eskin_ralston}
(\mu_1+i\mu_2)\cdot\int_{\R^n} W(x) e^{ix\cdot\xi} e^{\phi(x)} dx=(\mu_1+i\mu_2)\cdot\int_{\R^n} W(x) e^{ix\cdot\xi} dx.
\end{equation}
\end{prop}

\begin{proof}

The statement of the proposition for $W\in C_0(\R^n, \C^n)$ is due to \cite{Eskin_Ralston_1995}, with similar ideas appearing in \cite{Sun_1993}. See also \cite[Lemma 6.2]{Salo_2006}. For the completeness and convenience of the reader, we shall give a complete proof of the proposition here.

Assume first that  $W\in C^\infty_0(\R^n, \C^n)$. Then  by Lemma \ref{lem_salo_1} we have
\begin{equation}
\label{eq_3_15}
\phi=N_{\mu_1+i\mu_2}^{-1}(- i(\mu_1+i\mu_2)\cdot W)\in C^\infty(\R^n).
\end{equation}
We can always assume that $\mu_1=(1,0,\dots, 0)$ and $\mu_2=(0,1,0,\dots, 0)$, so that $\xi=(0,0,\xi'')$, $\xi''\in\R^{n-2}$, and therefore,
\[
(\p_{x_1}+i\p_{x_2})\phi=-i(\mu_1+i\mu_2)\cdot W\quad \textrm{in}\quad \R^n.
\]
Hence, writing $x=(x',x'')$, $x'=(x_1,x_2)$, $x''\in \R^{n-2}$, we get 
\begin{align*}
(\mu_1+i\mu_2)\cdot\int_{\R^n} W(x) e^{ix\cdot\xi} e^{\phi(x)} dx&=i\int_{\R^n}e^{ix''\cdot\xi''}e^{\phi(x)}(\p_{x_1}+i\p_{x_2})\phi(x) dx\\
&=i\int_{\R^{n-2}}e^{ix''\cdot\xi''} h(x'')dx'',
\end{align*}
where 
\begin{align*}
h(x'')=\int_{\R^2} (\p_{x_1}+i\p_{x_2})e^{\phi(x)} dx'&=\lim_{R\to \infty} \int_{|x'|\le R} (\p_{x_1}+i\p_{x_2})e^{\phi(x)} dx'\\
&=\lim_{R\to \infty} \int_{|x'|= R}e^{\phi(x)}(\nu_1+i\nu_2)dS_R(x').
\end{align*}
Here $\nu=(\nu_1,\nu_2)$ is the unit outer normal to the circle $|x'|= R$, and we have used the Gauss theorem.  

It follows from \eqref{eq_3_15} that $|\phi(x',x'')|=\mathcal{O}(1/|x'|)$ as $|x'|\to \infty$.  Hence, we have 
\[
e^\phi=1+\phi+\mathcal{O}(|\phi|^2)=1+\phi+\mathcal{O}(|x'|^{-2})\quad \textrm{as}\quad  |x'|\to \infty. 
\]
Since
\begin{align*}
\int_{|x'|= R}(\nu_1+i\nu_2)dS_R(x')=\int_{|x'|\le R}(\p_{x_1}+i\p_{x_2}) (1)dx'=0,\\
\bigg|Ê\int_{|x'|= R}\mathcal{O}(|x'|^{-2})(\nu_1+i\nu_2)dS_R(x')\bigg|\le \mathcal{O}(R^{-1})\quad \textrm{as}\quad R\to \infty, 
\end{align*}
we obtain that 
\begin{align*}
h(x'')=\lim_{R\to \infty} \int_{|x'|= R}\phi(x)(\nu_1+i\nu_2)dS_R(x')&=\lim_{R\to \infty} \int_{|x'|\le  R} (\p_{x_1}+i\p_{x_2})\phi(x)dx'\\
&=-\int_{\R^2} i(\mu_1+i\mu_2)\cdot W(x)dx',
\end{align*}
which shows \eqref{eq_eskin_ralston} for $W\in C^\infty_0(\R^n, \C^n)$.

To prove \eqref{eq_eskin_ralston}   for $W\in (L^\infty\cap \mathcal{E}')(\R^n, \C^n)$, consider the regularizations $W_j=\chi_j*W\in C^\infty_0(\R^n)$. Here $\chi_j(x)=j^n\chi(jx)$ is the usual mollifier with $0\le\chi\in C^\infty_0(\R^n)$ such that $\int\chi dx=1$. Then  $W_j\to W$ in $L^2(\R^n)$ as $j\to \infty$ and 
\begin{equation}
\label{eq_W_j_norm}
\|W_j\|_{L^\infty(\R^n)}\le \|W\|_{L^\infty(\R^n)}\|\chi_j\|_{L^1(\R^n)}=\|W\|_{L^\infty(\R^n)},\quad j=1,2,\dots.
\end{equation}
Furthermore, there is a compact set $K\subset\subset \R^n$ such that $\supp(W_j),\supp(W)\subset K$, $j=1,2,\dots$.  

We set 
$\phi_j=N_{\mu_1+i\mu_2}^{-1}(- i(\mu_1+i\mu_2)\cdot W_j)\in C^\infty(\R^n)$. Then by Lemma \ref{lem_SU}, we know that $\phi_j\to \phi$ in $L^2_{\textrm{loc}}(\R^n)$ as $j\to \infty$.  Lemma \ref{lem_salo_1} together with the estimate \eqref{eq_W_j_norm} implies that 
\begin{equation}
\label{eq_W_j_norm_1}
\|\phi_j\|_{L^\infty(\R^n)}\le C \|W_j\|_{L^\infty(\R^n)}\le C\|W\|_{L^\infty(\R^n)},\quad j=1,2,\dots.
\end{equation}

For $j=1,2,\dots$, we have 
\begin{equation}
\label{eq_eskin_ralston_1}
(\mu_1+i\mu_2)\cdot\int_{K} W_j(x) e^{ix\cdot\xi} e^{\phi_j(x)} dx=(\mu_1+i\mu_2)\cdot\int_{K} W_j(x) e^{ix\cdot\xi} dx.
\end{equation}
The fact that the integral in right hand side of \eqref{eq_eskin_ralston_1} converges to the integral in the right hand side of \eqref{eq_eskin_ralston} as $j\to \infty$ follows from the estimate 
\[
\bigg|(\mu_1+i\mu_2)\cdot\int_{K} (W_j(x)-W(x)) e^{ix\cdot\xi} dx\bigg|\le C\|W_j-W\|_{L^2(K)}\to 0,\quad j\to \infty.
\]
In order to show that the integral in the left hand side of \eqref{eq_eskin_ralston_1} converges to the integral in the left hand side of \eqref{eq_eskin_ralston} as $j\to \infty$, we establish that $I_1+I_2\to 0$ as $j\to \infty$, where 
\begin{align*}
I_1:=(\mu_1+i\mu_2)\cdot\int_{K} (W_j(x)-W(x)) e^{ix\cdot\xi} e^{\phi_j(x)} dx,\\
I_2:= (\mu_1+i\mu_2)\cdot\int_{K} W(x) e^{ix\cdot\xi} (e^{\phi_j(x)}-e^{\phi(x)}) dx.
\end{align*}
Using \eqref{eq_W_j_norm_1}, we have
\[
|I_1|\le C e^{\|\phi_j\|_{L^{\infty}(\R^n)}}\int_K |W_j(x)-W(x)|dx\le  C\|W_j-W\|_{L^2(K)}\to 0,\quad j\to \infty.
\]
Using \eqref{eq_complex_exp} and \eqref{eq_W_j_norm_1}, we get 
\[
|I_2|\le C\|W\|_{L^\infty(\R^n)}\|e^{\phi_j(x)}-e^{\phi(x)}\|_{L^2(K)}\le C\|\phi_j-\phi\|_{L^2(K)}\to 0, \quad j\to \infty.
\]
Here we have also used that $\phi_j\to \phi$ in $L^2_{\textrm{loc}}(\R^n)$ as $j\to \infty$.
Hence, passing to the limit as $j\to \infty$ in \eqref{eq_eskin_ralston_1}, we  obtain the identity \eqref{eq_eskin_ralston}. The proof is complete. 
\end{proof}

By Proposition \ref{prop_eskin_ralston} we conclude from \eqref{eq_with_phases_R_n} and \eqref{eq_sum_phase} that 
\begin{equation}
\label{eq_without_phases}
(\mu_1+i\mu_2)\cdot\int_{\R^n} (A_1(x)-A_2(x)) e^{ix\cdot\xi} dx=0. 
\end{equation}

It follows from \eqref{eq_without_phases} that $\mu\cdot (\hat A_1(\xi)-\hat A_2(\xi))=0$ whenever $\mu,\xi\in\R^n$ are such that $\mu\cdot \xi=0$. Here $\hat A_j$ is the Fourier transform of $A_j$, $j=1,2$. Let $\mu_{jk}(\xi)=\xi_j e_k-\xi_k e_j$ for $j\ne k$, where $e_1,\dots, e_n$ is the standard basis of $\R^n$.  Then $\mu_{jk}(\xi)\cdot\xi=0$, and therefore, 
\[
\xi_j(\hat A_{1,k}(\xi)-\hat A_{2,k}(\xi))-\xi_k(\hat A_{1,j}(\xi)-\hat A_{2,j}(\xi))=0. 
\]
Hence, $\p_{x_j}(A_{1,k}-A_{2,k})-\p_{x_k}(A_{1,j}-A_{2,j})=0$ in $\R^n$ in the sense of distributions,  for $j\ne k$, and thus, $d(A_1-A_2)=0$ in $\R^n$.

Our next goal is to show that $q_1=q_2$ in $\Omega$. First, viewing $A_1-A_2$ as a $1$--current and using the Poincar\'e lemma for currents, we conclude that there is $\psi\in \mathcal{D}'(\R^n)$ such that 
$d\psi=A_1-A_2\in (L^\infty\cap\mathcal{E}')(\R^n)$ in $\R^n$, see \cite{Rham_book}.  It follows from \cite[Theorem 4.5.11]{Horm_book_1} that $\psi$ is continuous on $\R^n$, and since $\psi$ is constant near infinity, we have $\psi\in L^\infty(\R^n)$. Therefore, $\psi\in W^{1,\infty}(\R^n)$, and without loss of generality, we may assume that there is an open ball $B$ such that $\Omega\subset\subset B$ and $\supp(\psi)\subset B$. 

We want to add $\nabla \psi$ to the potential $A_2$ without changing the set of the Cauchy data for $L_{A_2,q_2}$ on the ball $B$. To that end, we shall need the following result, which is due to \cite[Lemma 4.2]{Salo_diss}.

\begin{prop}
\label{prop_Cauchy_data}
Let $\Omega, \Omega'\subset \R^n$ be  bounded open sets such that $\Omega\subset\subset \Omega'$.  Let $A_1,A_2\in L^\infty(\Omega',\C^n)$, and  $q_1,q_2\in L^\infty(\Omega',\C)$. Assume that 
\begin{equation}
\label{eq_equality_A_q}
A_1=A_2\quad\textrm{and}\quad q_1=q_2\quad \textrm{in}\quad \Omega'\setminus\Omega.
\end{equation}
If  $C_{A_1,q_1}=C_{A_2,q_2}$ then $C'_{A_1,q_1}=C'_{A_2,q_2}$, where $C'_{A_j,q_j}$ is the set of the Cauchy data for $L_{A_j,q_j}$ in $\Omega'$, $j=1,2$. 
\end{prop}

\begin{proof}
Let $u_1'\in H^1(\Omega')$ be a solution to $L_{A_1,q_1}u_1'=0$ in $\Omega'$ and let $u_1=u_1'|_{\Omega}\in H^1(\Omega)$.   As $C_{A_1,q_1}=C_{A_2,q_2}$, there exists $u_2\in H^1(\Omega)$
satisfying $L_{A_2,q_2}u_2=0$ in $\Omega$ such that 
\[
Tu_ 2=T u_1\quad\textrm{and}\quad N_{A_2,q_2}u_2=N_{A_1,q_1}u_1\quad \textrm{in}\quad \Omega. 
\]
In particular, $\varphi:=u_2-u_1\in H^1_0(\Omega)\subset H^1_0(\Omega')$. We define 
\[
u_2'=u_1'+\varphi\in H^1(\Omega'),
\]
so that $u_2'=u_2$ on $\Omega$.  
It follows that $Tu_2'=Tu_1'$ in $\Omega'$. 

Let us show now that $L_{A_2,q_2}u_2'=0$ in $\Omega'$. To that end, let $\psi\in C_0^\infty(\Omega')$, and write 
\begin{align*}
\langle L_{A_2,q_2}u_2',\psi \rangle_{\Omega'}  =&\int_{\Omega'}\bigg((\nabla u_1'+\nabla \varphi)\cdot\nabla \psi+ A_2\cdot (Du_1'+D\varphi)\psi\bigg) dx\\
&+ \int_{\Omega'} \bigg(-A_2(u_1'+\varphi)\cdot D\psi+ (A_2^2+q_2)(u_1'+ \varphi)\psi\bigg) dx.
\end{align*}
Using \eqref{eq_equality_A_q}, we have 
\begin{align*}
\langle L_{A_2,q_2}u_2',\psi \rangle_{\Omega'} & = \int_{\Omega}(\nabla u_2\cdot\nabla \psi+ A_2\cdot (Du_2)\psi-A_2u_2\cdot D\psi+ (A_2^2+q_2)u_2\psi)dx\\
&+\int_{\Omega'\setminus\Omega}(\nabla u_1'\cdot\nabla \psi+ A_1\cdot (Du_1')\psi-A_1u_1'\cdot D\psi+ (A_1^2+q_1)u_1'\psi)dx\\
&+\int_{\Omega'\setminus\Omega}(\nabla \varphi\cdot\nabla \psi+ A_1\cdot (D\varphi)\psi-A_1\varphi\cdot D\psi+ (A_1^2+q_1)\varphi\psi)dx.
\end{align*}
As $\varphi\in H^1_0(\Omega)$, we get
\[
\int_{\Omega'\setminus\Omega}(\nabla \varphi\cdot\nabla \psi+ A_1\cdot (D\varphi)\psi-A_1\varphi\cdot D\psi+ (A_1^2+q_1)\varphi\psi)dx=0. 
\]
This together with the fact $N_{A_2,q_2}u_2=N_{A_1,q_1}u_1$ in $\Omega$ implies  that 
\begin{align*}
\langle L_{A_2,q_2}u_2',\psi \rangle_{\Omega'} & = ( N_{A_2,q_2}u_2 ,[\psi|_\Omega] )_{\Omega} \\
&+\int_{\Omega'\setminus\Omega}(\nabla u_1'\cdot\nabla \psi+ A_1\cdot (Du_1')\psi-A_1u_1'\cdot D\psi+ (A_1^2+q_1)u_1'\psi)dx\\
&=\langle L_{A_1,q_1}u_1',\psi \rangle_{\Omega'}=0,
\end{align*}
which shows that $L_{A_2,q_2}u_2'=0$ in $\Omega'$. 

Arguing similarly, we see that $N_{A_2,q_2}u_2'=N_{A_1,q_1}u_1'$ in $\Omega'$, which allows us to conclude that $C'_{A_1,q_1}\subset C'_{A_2,q_2}$. The same argument in the other direction gives the claim.
\end{proof}

Let us extend $q_j$, $j=1,2$, to the open ball $B$ by defining $q_j=0$ in  $B\setminus\Omega$. Then 
using Proposition \ref{prop_Cauchy_data}, Lemma \ref{lem_obstruction_to_uniqueness} and the fact that $\psi|_{\p B}=0$, we obtain that 
\[
C'_{A_1,q_1}=C'_{A_2,q_2}=C'_{A_2+\nabla\psi,q_2}=C'_{A_1,q_2}.
\]
This implies the following integral identity,
\begin{equation}
\label{eq_int_identity_new_2}
\int_B(q_1-q_2)u_1\overline{u_2}dx=0,
\end{equation}
valid for any $u_1,u_2\in H^1(B)$ satisfying $L_{A_1,q_1}u_1=0$ in $B$ and $L_{\overline{A_1},\overline{q_2}}u_2=0$ in $B$, respectively.  

Let us choose $u_1$ and $u_2$ to be the complex geometric optics solutions in $B$, given by \eqref{eq_u_1} and \eqref{eq_u_2}, respectively.  In this case, it follows from \eqref{eq_sum_phase} that 
$\Phi_1^\sharp (\cdot,\mu_1+i\mu_2;h)+ \overline{\Phi_2^\sharp (\cdot,-\mu_1+i\mu_2;h)}$ converges to zero in $L^2_{\textrm{loc}}(\R^n)$ as $h\to 0$. 

Plugging $u_1$ and $u_2$ into \eqref{eq_int_identity_new_2} gives
\[
\int_B(q_1-q_2)e^{ix\cdot\xi} e^{\Phi_1^\sharp+\overline{\Phi_2^\sharp}}dx=-\int_B(q_1-q_2)e^{ix\cdot\xi} (e^{\Phi_1^\sharp}\overline{r_2}+r_1e^{\overline{\Phi_2^\sharp}}+r_1\overline{r_2})dx.
\]
Letting $h\to 0$, and using \eqref{eq_est_phi_1}, \eqref{eq_est_r_1}, \eqref{eq_est_phi_2}, and  \eqref{eq_est_r_2}, we get 
\[
\int_B(q_1-q_2)e^{ix\cdot\xi}dx=0,
\] 
and therefore, $q_1=q_2$ in $\Omega$. The proof of Theorem \ref{thm_main} is complete.

\section*{Acknowledgements}  

The research of K.K. is partially supported by the
Academy of Finland (project 255580).   The research of
G.U. is partially supported by the National Science Foundation.

\end{document}